\documentclass[12pt,a4paper,reqno]{amsart}
\usepackage{amssymb}
\usepackage{dcpic,pictex}
\usepackage{multirow}
\usepackage{graphicx}
\usepackage{slashbox}
\usepackage{cite}

\usepackage{hyperref}

\newcommand{\CC}{{_{0}C^1[0,T]}}

\newcommand{\va}{\varphi}
\newcommand{\ppp}{\partial}
\newcommand{\ooo}{\overline}
\newcommand{\sumij}{\sum_{i,j=1}^n}

\newcommand{\fdif}{\partial_t^{\alpha}}
\newcommand{\OOO}{\Omega}

\newcommand{\N}{\mathbf{N}}

\newtheorem{theorem}{Theorem}

\newtheorem{corollary}[theorem]{Corollary}

\newtheorem{lemma}[theorem]{Lemma}

\begin{document}
\title[On the maximum principle for a time-fractional diffusion equation] {On the maximum principle for a time-fractional diffusion equation \footnote{This E-Print is a reproduction with a different layout of the paper published in
\emph{Fract. Calc. Appl. Anal.} \textbf{20} (2017), 1131--1145, DOI: https://doi.org/10.1515/fca-2017-0060 }}


\author{Yuri Luchko}
\address{Department of Mathematics, Physics, and Chemistry, 
Beuth University of Applied Sciences Berlin, 
Luxemburger Str. 10,
13353 -- Berlin, GERMANY }
\email{luchko@beuth-hochschule.de}

\author{Masahiro Yamamoto}
\address{Dept. of Mathematical Sciences, 
The University of Tokyo, 
Komaba, Meguro, 
153 -- Tokyo, JAPAN}
\email{myama@ms.u-tokyo.ac.jp}




 \begin{abstract}

In this paper, we discuss the maximum principle for a time-fractional diffusion equation
\vskip -10pt
$$
\fdif u(x,t) = \sum_{i,j=1}^n \ppp_i(a_{ij}(x)\ppp_j u(x,t))
+ c(x)u(x,t) + F(x,t),\ t>0,\ x \in \Omega \subset {\mathbb R}^n,
$$
with the Caputo time-derivative of the order $\alpha \in (0,1)$ in the case
of the homogeneous Dirichlet boundary condition.
Compared to the already published results, our findings have two important
special features.  First, we derive a maximum principle for a suitably defined
weak solution in the fractional Sobolev spaces, not for the strong solution.
Second, for the non-negative source functions $F = F(x,t)$
we prove the non-negativity of the weak solution to the problem under consideration without
any restrictions on the sign of the coefficient $c=c(x)$ by the derivative of order zero in the spatial differential operator.  
Moreover, we prove the monotonicity of the solution with respect to the coefficient $c=c(x)$.

 \end{abstract}
  


 \vspace*{-4pt}


\maketitle


\section{Introduction}\label{sec1}

\setcounter{section}{1}
\setcounter{equation}{0}\setcounter{theorem}{0}

During the last few decades various fractional generalizations of the classical diffusion equation were introduced and intensely discussed both in the mathematical literature and in different applications, say, as models for the so called anomalous diffusion (see e.g. \cite{Met14} and the numerous references therein).
The mathematical theory of the fractional diffusion equations is nowadays
under remarkable development, but still it is not as complete as the theory of the partial differential equations of the parabolic type.

One of the recent research topics in this theory is studying the analogies of the maximum principles known for the parabolic and elliptic types of partial differential equations as well as their applications to analysis of solutions to the boundary- or initial-boundary-value problems for the fractional partial differential equations.
The first publications that should  be mentioned in this connection are the papers \cite{Koc1} and \cite{Koc2},  where a kind of a maximum principle was employed for analysis of some fractional partial differential equations without an explicit formulation of this principle.
In \cite{luchko-1}, a weak maximum principle for a single-term time-fractional diffusion equation with the Caputo fractional derivative was formulated and proved for the first time. In \cite{luchko-2}, this principle was applied
for an a priori estimate for solutions to the initial-boundary-value problems
for a multi-dimensional time-fractional diffusion equation.
The weak maximum principles for multi-term time-fractional diffusion equations
and time-fractional diffusion equations with the Caputo fractional derivatives
of the distributed orders were introduced and applied in \cite{luc-mt} and \cite{luc-do}, respectively. In \cite{Liu}, a strong maximum principle for time-fractional diffusion equations with the Caputo derivatives was established and applied for proving a uniqueness result for a related inverse source problem of determination of the temporal component of the source equation term. 
In \cite{Al1}, \cite{Al2}, and \cite{Al3} the maximum principles for single-,\break  multi-term, and distributed order  fractional diffusion equations with the Riemann-Liouville fractional derivatives, respectively,  were proved and employed for analysis of solutions to the initial-boundary-value problems for linear and non-linear time-fractional partial differential equations. A maximum principle for multi-term time-space fractional differential equations with the modified Riesz space-fractional derivative in the Caputo sense was introduced and employed in \cite{Ye}.  In \cite{Liu1},
a maximum principle for multi-term time-space variable-order fractional differential equations with the Riesz-Caputo fractional derivatives was proved and applied for analysis of these equations.
Finally, we mention a very recent paper \cite{luc-yam}, where a weak maximum principle for a general
time-fractional diffusion equation which was introduced in \cite{Koch11},
was derived and employed for proving the uniqueness of both the strong and the weak solutions to
the initial-boundary-value problem for this equation. The general
time-fractional diffusion equation contains both single- and
multi-term time-fractional diffusion equations as well as time-fractional
diffusion equation of the distributed order among its particular cases and is a new object in fractional calculus worth to be investigated in detail.

\vskip 2pt 

In this paper, we revisit the weak maximum principle for the time-fractional
diffusion equation
\vskip -10pt
$$
\fdif u(x,t) = \sum_{i,j=1}^n \ppp_i(a_{ij}(x)\ppp_ju(x,t))
+ c(x)u(x,t) + F(x,t),\ t>0,\ x \in \Omega \subset {\mathbb R}^n
$$
\vskip -3pt \noindent
with the Caputo time-derivative of the order $\alpha \in (0,1)$ and prove it for a suitably defined weak solution in the fractional Sobolev spaces and without any restrictions on the sign of the coefficient $c=c(x)$.

The rest of this paper is organized as follows. In Section \ref{sec2},  the problem
that we are dealing with as well as our results are formulated. Section \ref{sec3} is devoted to a proof  of a key lemma that is a basis for the proofs of all other results. The lemma asserts that the solution mapping $\{a, F\} \longrightarrow
u_{a,F}$ ($a$ and $F$ denote an initial condition and a source function of the problem under consideration, respectively, and $u_{a,F}$ denotes its weak solution)  preserves its sign.  In Section \ref{sec4}, the key lemma and the fixed point theorem are employed to prove the maximum and comparison principles and some of their corollaries.  Finally, some conclusions and remarks are formulated in the last section.


\section{Problem formulation and main results} 
\label{sec2}

\setcounter{section}{2}
\setcounter{equation}{0}\setcounter{theorem}{0}

In this paper, we deal with the following initial-boundary-value problem for the single-term time-fractional diffusion equation
\vskip - 10pt
$$
\fdif u(x,t) = \sum_{i,j=1}^n \ppp_i(a_{ij}(x)\ppp_ju(x,t))
+ c(x)u(x,t) + F(x,t), \ x \in \Omega \subset {\mathbb R}^n, \thinspace t>0, \eqno{(2.1)}
$$
$$
u(x,t) = 0, \qquad x\in \ppp\Omega, \thinspace t>0, \eqno{(2.2)}
$$
$$
u(x,0) = a(x), \qquad x\in \Omega                 \eqno{(2.3)}
$$
with $0 < \alpha < 1$ and in a bounded domain $\Omega $ with a smooth boundary
$\ppp\Omega$. In what follows, we always suppose that
$a_{ij} \equiv  a_{ji} \in C^1(\ooo{\Omega})$, $1\le i,j \le n$, $c \in
C(\ooo{\Omega})$, and there exists a constant $\mu_0 > 0$ such that
$\sum_{i,j=1} a_{ij}(x)\xi_i\xi_j \ge \mu_0 \sum_{i=1}^n \xi_i^2$
for all $x \in \ooo{\Omega}$ and $\xi_1, ..., \xi_n \in \mathbb{R}$, i.e., that the spatial differential operator in equation (2.1) is a uniformly elliptic one.
The fractional derivative $\fdif u$ in (2.1) is defined in  the Caputo sense by
$$
\fdif u(x,t) = \frac{1}{\Gamma(1-\alpha)}\int^t_0
(t-s)^{-\alpha}\ppp_su(x,s) ds, \quad x\in \Omega, \thinspace
t > 0, \quad u \in C^1[0,T].
$$
In \cite{GLY}, the Caputo fractional derivative $\fdif$ was extended to an operator defined on
the closure $H_{\alpha}(0,T)$ of $\CC := \{ u \in C^1[0,T];\thinspace
u(0) = 0\}$ in the fractional Sobolev space $H^{\alpha}(\OOO)$. In what follows,
we regard $\fdif u$ in (2.1) as this extension with the domain
$H_{\alpha}(0,T)$ (see \cite{GLY} for details).
Thus we interpret the problem (2.1) - (2.3) as the fractional diffusion equation (2.1) subject to the inclusions
$$\left\{ \begin{array}{l}
u(\cdot,t) \in H^1_0(\OOO), \quad t>0, \\
u(x,\cdot) - a(x) \in H_{\alpha}(0,T), \quad x \in \OOO.
\end{array}\right.
                                             \eqno{(2.4)}
$$
According to the results presented in \cite{GLY}, for any initial condition $a \in L^2(\OOO)$ and
any source function $F \in L^2(\OOO\times (0,T))$,
there exists a unique weak solution
$u_{a,F} \in L^2(0,T;H^2(\OOO)\cap H^1_0(\OOO)) \cap
H_{\alpha}(0,T; L^2(\OOO))$ to (2.1) subject to the inclusions (2.4).
For $\frac{1}{2} < \alpha < 1$, in view of the Sobolev embedding the solution $u$ belongs to the functional space $C([0,T];L^2(\OOO))$
and satisfies the initial condition (2.3) in the $L^2$-sense.

The focus of this paper is on the weak maximum principle for the equation (2.1), which
says that the inequalities $F(x,t)\ge 0,\ (x,t)\in \Omega \times (0,T)$ and
$a(x)\ge 0,\ x \in \Omega$ yield the inequality $u(x,t) \ge 0,\ (x,t) \in \Omega \times (0,T)$ for
the weak solution to the initial-boundary value problem (2.1)-(2.3) defined as in \cite{GLY}.

A maximum principle for the strong solution to the initial-boundary value problem (2.1)-(2.3) was
first proved  in \cite{luchko-1}  under the assumption that
\vskip -12pt
$$
c(x) \le 0, \quad x \in \ooo{\Omega}.
$$
Moreover, the case of an inhomogeneous Dirichlet boundary
condition
$$
u\vert_{\ppp\Omega\times (0,T)} = b(x,t)
$$
was also considered in \cite{luchko-1}, but in this paper, for the sake of technical
simplicity, we assume the homogeneous boundary condition $u\vert_{\ppp\Omega\times (0,T)} = 0$ 
although our method can be applied for the case of an inhomogeneous  boundary condition, too.

\smallskip

The main result of this paper is a proof of the weak maximum principle  for the equation (2.1) with any $c \in C(\ooo{\OOO})$
without the non-negativity condition $c(x)\le 0,\  x \in \ooo{\Omega}$. As it is known, the weak maximum principle for the partial differential equations of the parabolic type is valid without any condition on the sign of the coefficient $c=c(x)$ (see e.g. \cite{MP1} or \cite{walter}). The proof of this fact uses the properties of the exponential function and reduces the case of a bounded coefficient $c=c(x)$ to the case of a non-negative coefficient. This technique does not work in the case of the fractional diffusion equation (2.1) and thus we were forced to invent a new and more complicated proof method.

Let us now denote the solution to the initial-boundary value problem (2.1)-(2.3)  defined as in \cite{GLY} by $u_{a,F}$ and   formulate our results.

\vspace*{-3pt} 

\begin{theorem}
\label{t1}
Let $a \in L^2(\OOO)$ and $F \in L^2(\OOO\times (0,T))$.
If $F(x,t) \ge 0$ a.e. (almost everywhere) in $\OOO\times (0,T)$ and
$a(x) \ge 0$ a.e. in $\OOO$, then
$u_{a,F}(x,t) \ge 0$ a.e. in $\OOO\times (0,T)$.
\end{theorem}  

\vspace*{-2pt}

Let us mention that in \cite{luchko-1}  the maximum principle was stated
pointwise (i.e., for all points
from $\ooo{\OOO} \times [0,T]$) for the strong solution under the assumption that
$c(x)\le 0,\ x \in \OOO$. In Theorem \ref{t1},
the maximum principle is formulated for the weak solution and our proof
is based on the fixed point theorem and the property that the
solution mapping $\{a, F\} \longrightarrow u_{a,F}$  preserves its sign on the set of the
weak solutions and thus the non-negativity of the solution is
valid almost everywhere and not pointwise.

Theorem \ref{t1} immediately yields the following comparison property:

\vspace*{-3pt} 

\begin{corollary}
\label{c1}
Let $a_1, a_2  \in L^2(\OOO)$ and $F_1, F_2 \in L^2(\OOO\times (0,T))$
satisfy the inequalities $a_1(x) \ge a_2(x)$ a.e. in $\OOO$ and $F_1(x,t) \ge F_2(x,t)$ a.e. in
$\OOO\times (0,T)$, respectively.  Then $u_{a_1,F_1}(x,t) \ge u_{a_2,F_2}(x,t)$ a.e. in
$\OOO\times (0,T)$.
\end{corollary}

\vspace*{-2pt} 

Corollary \ref{c1} can be employed among other things to remove the condition $c(x)\le 0$ from the formulation of the strong maximum principle for the fractional diffusion equation that was derived in \cite{Liu}.

Let us now fix a source function $F = F(x,t) \ge 0$ and an initial condition $a= a(x) \ge 0$ and denote by $u_c = u_c(x,t)$ the weak solution to the initial-boundary-value problem (2.1)-(2.3) with
the coefficient $c=c(x)$.
Then the following comparison property is valid:

\vspace*{-3pt} 

\begin{theorem}
\label{t2}
Let $c_1, c_2 \in C(\ooo{\OOO})$ satisfy the inequality $c_1(x) \ge c_2(x)$ in $\OOO$.
Then $u_{c_1}(x,t) \ge u_{c_2}(x,t)$ in $\OOO \times (0,T)$.
\end{theorem}

\vspace*{-3pt} 

One of the useful consequences from Theorem \ref{t2} is given in the following statement:

\vspace*{-3pt} 

\begin{corollary}
\label{c2}
Let $n\le 3$ ($\Omega \subset \mathbb{R}^n$), the initial condition $a \in L^2(\OOO)$
satisfy the inequality $a(x) \ge 0$ a.e. in $\OOO$, $a\not\equiv 0$, and the source function be identically equal to zero, i.e.,
$F(x,t)\equiv 0,\  x \in \Omega, \ t>0$.
Then the weak solution $u$ to the initial-boundary-value problem (2.1)-(2.3) satisfies the inclusion $u \in C((0,T]; C(\ooo{\OOO}))$ and for each $x \in \Omega$
the set $\{ t:\, t >0 \wedge \thinspace u(x,t) \le 0\}$ is at most a finite set.
\end{corollary}

\section{Solution mapping and its properties} 
\label{sec3}

\setcounter{section}{3}
\setcounter{equation}{0}\setcounter{theorem}{0}

Let us define an operator $A$ in $L^2(\OOO)$ by the relation
\vskip -10pt
$$
(Av)(x) = -\sumij \ppp_i(a_{ij}(x)\ppp_jv(x)) - c(x)v(x), \quad
x\in \OOO
$$
\vskip -3pt \noindent
with $\mathcal{D}(A) = H^2(\OOO) \cap H^1_0(\OOO)$ and  assume that the inequality
$$
c(x) < 0, \qquad  x\in \ooo{\OOO}    \eqno{(3.1)}
$$
\vskip -3pt \noindent
and the conditions on the coefficients $a_{ij}$ formulated at the beginning of Section \ref{sec2} are satisfied.
By $\Vert \cdot\Vert$ and $(\cdot,\cdot)$ we denote the standard norm and
the scalar product  in $L^2(\OOO)$, respectively.
Then it is known that the operator $A$ is self-adjoint and positive definite in
$L^2(\OOO)$  and therefore its spectrum consists of discrete positive eigenvalues $0 < \mu_1 \le \mu_2 \le
\cdots$
which are numbered according to their multiplicities and  $\mu_n \to +\infty$ as $n\to +\infty$.
Let $\va_n$ be an eigenvector corresponding to the eigenvalue $\mu_n$ such that
$(\va_n, \va_m) = 0$ if $n \ne m$ and $(\va_n,\va_n) = 1$.
Then it is also known that the system $\{ \va_n\}_{n\in \N}$ of the eigenvectors forms an
orthonormal basis in $L^2(\OOO)$ and for any $\gamma\ge 0$ we can define
the fractional powers $A^{\gamma}$ of the operator $A$ by the following relation (see e.g. \cite{Pa}):
\vskip -10pt
$$
A^{\gamma}v = \sum_{n=1}^{\infty} \mu_n^{\gamma} (v,\va_n)\va_n,
$$
\vskip - 4pt \noindent
where
\vskip - 14pt
$$
v \in \mathcal{D}(A^{\gamma})
:= \left\{ v\in L^2(\OOO): \thinspace
\sum_{n=1}^{\infty} \mu_n^{2\gamma} (v,\va_n)^2 < \infty\right\}
$$
\vskip -3pt \noindent
and
\vskip - 14pt
$$
\Vert A^{\gamma}v\Vert = \left( \sum_{n=1}^{\infty}
\mu_n^{2\gamma} (v,\va_n)^2 \right)^{\frac{1}{2}}.
$$
Let us define two other operators, $S(t)$ and $K(t)$, by the relations
\vskip -10pt
$$
S(t)a = \sum_{n=1}^{\infty} E_{\alpha,1}(-\mu_n t^{\alpha})
(a,\va_n)\va_n, \quad a\in L^2(\OOO), \thinspace t>0  \eqno{(3.2)}
$$
\vskip -4pt \noindent
and
\vskip -12pt
$$
K(t)a = \sum_{n=1}^{\infty} t^{\alpha-1}E_{\alpha,\alpha}(-\mu_n t^{\alpha})
(a,\va_n)\va_n, \quad a\in L^2(\OOO), \thinspace t>0,
                                                     \eqno{(3.3)}
$$
\vskip -2pt ]noindent
where $E_{\alpha,\beta}(z)$ denotes the Mittag-Leffler function defined by a convergent series as follows:
$$
E_{\alpha,\beta}(z) = \sum_{k=0}^\infty \frac{z^k}{\Gamma(\alpha\, k + \beta)},\ \alpha >0,\ \beta \in \mathbb{C},\ z \in \mathbb{C}.
$$
It follows directly from the definitions given above that
$A^{\gamma}K(t)a = K(t)A^{\gamma}a$
and $A^{\gamma}S(t)a = S(t)A^{\gamma}a$ for $a \in \mathcal{D}(A^{\gamma})$.
Moreover, the following norm estimates  were proved in \cite{SY}:
\vskip -10pt
$$\left\{ \begin{array}{l}
\Vert S(t)a\Vert \le C\Vert a\Vert, \\
\Vert A^{\gamma}K(t)a\Vert \le Ct^{\alpha(1-\gamma)-1}
\Vert a\Vert, \quad a \in L^2(\OOO), \thinspace t > 0, \thinspace
0 \le \gamma \le 1.
\end{array}\right.                    \eqno{(3.4)}
$$
To shorten the notations and focus on the dependence of the time variable $t$, henceforth we sometimes omit the variable  $x$ in the functions of two variables, $t$ and $x$, and write simply
$u(t) = u(\cdot,t)$, $F(t) = F(\cdot,t)$, $a = a(\cdot)$, etc.
As it was shown in \cite{GLY}, under the condition (3.1) and for $F \in L^2(\OOO\times (0,T))$ and $a \in L^2(\OOO)$ the weak solution $u$ to the initial-boundary-value problem (2.1)-(2.3)
satisfies the inclusion $u \in H^{\alpha}(0,T;L^2(\Omega)) \cap
L^2(0,T;H^2(\OOO) \cap H^1_0(\Omega))$ and can be represented by the formula
\vskip - 10pt
$$
u(t) = S(t)a + \int^t_0 K(t-s)F(s) ds =: L(a,F)(t), \quad t > 0.
                                           \eqno{(3.5)}
$$
Now we paraphrase the result first presented in \cite{luchko-1} and formulate it as our key lemma.

\vspace*{-3pt} 

\begin{lemma} 
\label{l1}
Let the condition (3.1) on the coefficient $c=c(x)$ hold true. If the initial condition $a \in L^2(\OOO)$ and the source function $F \in L^2(\OOO\times (0,T))$
satisfy the inequalities
$$
a(x) \ge 0, \quad F(x,t)\ge 0 \quad \mbox{for almost all }\ x \in \OOO \ \mbox{ and } \
0 < t < T,                  \eqno{(3.6)}
$$
\vskip -3pt \noindent
then
\vskip -10pt
$$
L(a,F)(t) \ge 0, \quad \mbox{for almost all } x \in \OOO \mbox{ and }
0 < t < T.                  \eqno{(3.7)}
$$
\end{lemma}  

\begin{proof}
We start with proving the statement that under the conditions (3.1) and (3.6) the inequality (3.7) holds true for $a \in C^{\infty}_0(\OOO)$ and $F \in C^{\infty}_0(\OOO\times (0,T))$. First we  prove the inclusions
\vskip -10pt
$$
L(a,F) \in C^1((0,T]; C^2(\ooo{\OOO})) \cap C(\ooo{\OOO}\times
[0,T])                      \eqno{(3.8)}
$$
\vskip -3pt \noindent
and
\vskip -12pt
$$
\ppp_tL(a,F) \in L^1(0,T;L^2(\OOO)).    \eqno{(3.9)}
$$
Since $A^{\gamma}S(t)a = S(t)A^{\gamma}a$ for all $\gamma > 0$ and
$a \in C^{\infty}_0(\OOO) \subset \mathcal{D}(A^{\gamma})$,
the part $S(t)a$ of the operator $L(a,F)$ satisfies the inclusion (3.8) (see Corollary 2.6 in \cite{SY}).
Moreover, by differentiation of the series that defines the Mittag-Leffler function we have the relation
\vskip -10pt
$$
S'(t)a = \sum_{n=1}^{\infty} (a,\va_n)\left( \frac{d}{dt}
E_{\alpha,1}(-\mu_n t^{\alpha})\right)\va_n 
$$
\vspace*{-6pt} 
$$
=\, -t^{\alpha-1}\sum_{n=1}^{\infty} \mu_n(a,\va_n)E_{\alpha,\alpha}
(-\mu_n t^{\alpha})\va_n .
$$
Using the known asymptotics of the Mittag-Leffler function, we readily get the following norm estimates:
\vskip -13pt
$$
\Vert S'(t)a \Vert \le t^{\alpha-1}\left(
\sum_{n=1}^{\infty} \mu_n^2(a,\va_n)^2\frac{C}{(1+\mu_nt^{\alpha})^2}
\right)^{\frac{1}{2}}
\le Ct^{\alpha-1}\Vert a\Vert_{H^2(\OOO)}, \quad t > 0
$$
and thus the estimate
\vskip -10pt
$$
\int^T_0 \Vert S'(t)a\Vert dt < \infty,
$$
which verifies that $S(t)a$ satisfies the inclusion (3.9).

Next we have to prove that the second part of the operator $L(a,F)$, the function
$w(t) := \int^t_0 K(t-s)F(s) ds$, satisfies the inclusions (3.8) and (3.9), too.
Because the Laplace convolution is commutative, we can represent it as $w(t) = \int^t_0 K(s)F(t-s) ds$.
Due to the inclusion $F \in C^{\infty}_0(\OOO\times (0,T))$, we have
\vskip - 13pt
$$
w'(t) = K(t)F(0) + \int^t_0 K(s)F'(t-s) ds
= \int^t_0 K(s)F'(t-s) ds
$$
\vskip -2pt \noindent
and then we arrive at the representation
\vskip -10pt
$$
A^{\gamma}w'(t) = \int^t_0 K(s)A^{\gamma}F'(t-s) ds, \quad t > 0.
$$
\vskip -2pt \noindent
By employing the norm estimate (3.4), for any $\gamma > 0$ we have the estimates
\begin{align*}
&\Vert A^{\gamma}w'(t) \Vert \le C\int^t_0 s^{\alpha-1}
\Vert A^{\gamma}F'(t-s) \Vert ds\\
\le & \, C\max_{0\le t \le T} \Vert A^{\gamma}F'(t)\Vert
\int^t_0 s^{\alpha-1} ds < \infty.
\end{align*}
For a sufficiently large $\gamma > 0$, we then apply the Sobolev embedding
theorem, and thus arrive at the inclusion
$$
w' \in C((0,T]; C^2(\ooo{\OOO}))
$$
that implies that the function $w=w(t)$ satisfies both the inclusion (3.8) and the inclusion (3.9).

The inclusions (3.8) and (3.9) mean that in the case under consideration  the weak solution $u=u(t)$ can be interpreted as a strong solution and we are now in position to apply the maximum principle for the strong solution from  \cite{luchko-1} to the solution $u(t) = L(a,F)(t)$ and thus conclude that the inequality (3.7) holds true if
$a \in C^{\infty}_0(\OOO)$ and $F \in C^{\infty}_0(\OOO\times (0,T))$.

To complete the proof of the lemma, let us suppose that $a \in L^2(\OOO)$ and
$F \in L^2(\OOO\times (0,T))$ be arbitrarily given.
Then we can choose the sequences
$a_n \in C^{\infty}_0(\OOO)$ and $F_n \in C^{\infty}_0(\OOO\times (0,T))$
such that $a_n \to a$ in $L^2(\OOO)$ and
$F_n \to F$ in $L^2(\OOO\times (0,T))$.
As we already proved,
$$
L(a_n,F_n) \ge 0 \quad \mbox{a.e. in } \OOO \times (0,T), n\in \mathbb{N}.
                                     \eqno{(3.10)}
$$
Moreover, it was shown in \cite{GLY}  that
$L(a_n,F_n) \to L(a,F)$ as $n \to \infty$ in
$L^2(0,T;H^2(\OOO)) \cap H^{\alpha}(0,T;L^2(\OOO))$.
Hence the inequality (3.10) yields the inequality $L(a,F)\ge 0$ almost everywhere in $\OOO\times (0,T)$ and
the proof of the lemma is completed.
\end{proof}  

\vspace*{-4pt} 

\section{Proofs of the main results} 
\label{sec4}

\setcounter{section}{4}
\setcounter{equation}{0}\setcounter{theorem}{0}

In this section, the proofs of the main results stated
in Section \ref{sec2} are presented.  For the proofs, the fixed point theorem and Lemma \ref{l1} formulated and proved in the previous section play a decisive role.

We start with a proof of Theorem \ref{t1}.

\begin{proof}
Let us set $M= \max_{x\in \ooo{\OOO}}\vert c(x)\vert$ and introduce an auxiliary function
$c_0(x) = c(x) - (M+1)$.  Then $c_0(x) < 0$ on $\ooo{\OOO}$ and the boundary-value problem (2.1)-(2.3) can be rewritten in the form
$$\left\{ \begin{array}{l}
\fdif u(x,t) = \sum_{i,j=1}^n \ppp_i(a_{ij}(x)\ppp_ju(x,t))
+ c_0(x)u(x,t) + \\
(F(x,t)+(M+1)u(x,t)),\ \ x \in \Omega, \thinspace t>0, \\
u(x,t) = 0, \qquad x\in \ppp\Omega, \thinspace t>0, \\
u(x,0) = a(x), \qquad x\in \Omega.
\end{array}\right.
\eqno{(4.1)}
$$
The weak solution $u$ to the equation from (4.1) with the coefficient $c_0$ by the unknown function can be represented by the formula (3.5) via the operator $L(a,F)$ and we thus arrive at the integral equation for the solution $u$ in the form
\vskip - 12pt
$$
u(t) = L(a, F+(M+1)u)(t), \quad t > 0.      \eqno{(4.2)}
$$
\vskip -3pt \noindent
Now the fixed point theorem technique is applied to analyze the equation (4.2). Let us consider  a sequence of functions $u_n$ defined as follows:
$$
u_0 = 0, \quad u_{n+1} = L(a,F+(M+1)u_n), \quad
n=0,1,2,\dots \ .                                      \eqno{(4.3)}
$$
We first prove that the sequence $u_n$ is convergent in $L^2(\OOO\times (0,T))$.
In fact, denoting  $u_{n+1} - u_n,\ n =0,1,2,\dots $ by $d_{n+1}$  we immediately get the representation
\vskip - 12pt
$$
d_1(t) = u_1(t), \quad d_{n+1}(t) = \int^t_0 K(t-s)d_n(s) ds, \quad n=1,2,\dots
\ .
$$
Let us set $M_0 = \max_{0\le t\le T}\Vert u_1(t)\Vert$.  It follows from the norm estimates (3.4) that
\vskip - 12pt
$$
\Vert d_{n+1}(t)\Vert \le C\int^t_0 (t-s)^{\alpha-1}\Vert d_n(s)\Vert ds.
$$
Thus we can estimate the norm of $d_2(t)$ as follows:
\vskip -10pt
$$
\Vert d_2(t)\Vert \le CM_0\frac{t^{\alpha}}{\alpha}.
$$
For the norm of $d_3(t)$ we get the estimates
\begin{align*}
& \Vert d_3(t)\Vert \le C\int^t_0 (t-s)^{\alpha-1} \frac{CM_0s^{\alpha}}
{\alpha} ds\\
=& \, C\frac{CM_0}{\alpha}\frac{\Gamma(\alpha)\Gamma(\alpha+1)}
{\Gamma(2\alpha+1)}t^{2\alpha}
= \frac{(C\Gamma(\alpha))^2M_0}{\Gamma(2\alpha+1)}t^{2\alpha}.
\end{align*}
Proceeding as above and applying the principle of mathematical induction we finally obtain the norm estimate
$$
\Vert d_n(t)\Vert \le \frac{(C\Gamma(\alpha))^{n-1}M_0}
{\Gamma((n-1)\alpha+1)}t^{(n-1)\alpha}, \quad 0 < t < T,\ n \in \mathbb{N}.
$$
\vskip -3pt \noindent
Hence
\vskip -10pt
$$
\max_{0\le t \le T}\Vert d_n(t)\Vert
\le \frac{(C\Gamma(\alpha)T^{\alpha})^{n-1}M_0}
{\Gamma((n-1)\alpha+1)}, \quad n \in \mathbb{N}.
$$
Because $u_n - u_0 = u_n = \sum_{k=1}^n d_k$, let us investigate the convergence of the series $\sum_{k=1}^\infty d_k$.
By the quotient convergence rule and using the known asymptotic behavior of the Gamma function, we get
$$
\lim_{k\to\infty} \left(
\frac{(C\Gamma(\alpha)T^{\alpha})^k M_0}
{\Gamma(k\alpha+1)}\right)
\left( \frac{(C\Gamma(\alpha)T^{\alpha})^{k-1}M_0}
{\Gamma((k-1)\alpha+1)}\right)^{-1} 
$$
\vspace*{-3pt} 
$$
= \lim_{k\to\infty} C\Gamma(\alpha)T^{\alpha}
\frac{\Gamma((k-1)\alpha+1)}
{\Gamma(k\alpha+1)}
< 1,
$$
so that the series $\sum_{k=1}^{\infty} d_k(t)$ and thus the sequence $u_n$ are both convergent in $C([0,T];L^2(\OOO))$.  According to construction of the sequence
$u_n$, it converges to the fixed point of the integral equation  (4.2), i.e., to the weak solution $u$ of the initial-boundary-value problem (4.1).

Let us now show the non-negativity of $u$. Because of the condition (4.3) and taking into account the inequalities $a(x) >0, \ F(x,t) \ge 0$, Lemma \ref{l1} yields the inequality
$u_1(x,t) \ge 0$ in $\OOO\times (0,T)$.  Then $F(x,t)+(M+1)u_1(x,t) \ge 0$ in
$\OOO\times (0,T)$ and we can apply Lemma \ref{l1} to the solution representation (4.3) with $n=1$ and obtain the inequality
$u_2(x,t) \ge 0$ in $\OOO\times (0,T)$.  Repeating these arguments, we arrive at the inequality
$u_n(x,t) \ge 0$ in $\OOO\times (0,T)$ for all $n \in \mathbb{N}$.
Since $u_n \to u$ in $C([0,T];L^2(\OOO))$, the inequality $u(x,t)\ge 0$ holds true in
$\OOO\times (0,T)$ for the weak solution $u$ of the of the initial-boundary-value problem (4.1), too.  The proof of Theorem \ref{t1} is completed.
\end{proof}

We proceed with a proof of Corollary \ref{c1}.

\begin{proof}  
Let us denote $u_{a_1,F_1} - u_{a_2,F_2}$ by $y$. The function $y = y(x,t)$ is thus the weak solution to the following initial-boundary-value problem
$$\left\{ \begin{array}{l}
\fdif y(x,t) = \sum_{i,j=1}^n \ppp_i(a_{ij}(x)\ppp_jy(x,t))
+ F_1(x,t) - F_2(x,t), \ x \in \Omega, \thinspace t>0, \\
y(x,t) = 0, \qquad x\in \ppp\Omega, \thinspace t>0, \\
y(x,0) = a_1(x) - a_2(x), \qquad x\in \Omega.
\end{array}\right.
$$
According to the conditions posed in Corollary \ref{c1}, the inequalities $F_1(x,t)-F_2(x,t)\ge 0$ and $a_1(x) - a_2(x) \ge 0$ hold true in
$\OOO\times (0,T)$.  Then Theorem \ref{t1} implies that
$y(x,t) = u_{a_1,F_1}(x,t) - u_{a_2,F_2}(x,t) \ge 0$ almost everywhere in $\OOO \times (0,T)$.  Thus the proof of the
corollary is completed.
\end{proof}  

Now a proof of Theorem \ref{t2} is presented.

\begin{proof} 
First we introduce an auxiliary function $z = u_{c_1} - u_{c_2}$ that is the  weak solution to the following initial-boundary-value problem
$$\left\{ \begin{array}{l}
\fdif z(x,t) = \sum\limits_{i,j=1}^n \ppp_i(a_{ij}(x)\ppp_jz(x,t))
+ c_1(x)z(x,t) + \\
(c_1(x)-c_2(x))u_{c_2}(x,t), \ x \in \Omega, \thinspace t>0, \\
z(x,t) = 0, \qquad x\in \ppp\Omega, \thinspace t>0, \\
z(x,0) = 0, \qquad x\in \Omega.
\end{array}\right.
								\eqno{(4.4)}
$$
Because the inequalities $F(x,t)\ge 0$ and $a(x) \ge 0$ hold true in $\OOO\times (0,T)$ and in $\OOO$, respectively, Theorem \ref{t1} yields that $u_{c_2}(x,t) \ge 0$ in
$\OOO\times (0,T)$.  Hence $(c_1(x)-c_2(x))u_{c_2}(x,t) \ge 0$ in
$\OOO\times (0,T)$. Applying now
Theorem \ref{t1} to the initial-boundary-value problem (4.4) leads to the inequality $z(x,t) = u_{c_1}(x,t) - u_{c_2}(x,t) \ge 0$ in $\OOO\times (0,T)$ and the proof of Theorem \ref{t2} is completed.
\end{proof}

Finally, we give a proof of  Corollary \ref{c2}.

\begin{proof}  
Under the condition  $c(x)\le 0$, the statement of the corollary was already proved in \cite{Liu}. Let us show that the corollary holds true also without this condition.  According to \cite{SY}, the weak solution $u$ to the initial-boundary-value problem (2.1)-(2.3) belongs to the functional space $C((0,T];H^2(\Omega))$.  For $n\le 3$, the Sobolev
embedding theorem implies that $H^2(\OOO) \subset C(\ooo{\OOO})$ and thus we get the inclusion
$u \in C((0,T]; C(\ooo{\OOO}))$.

Let us denote by $u$ the weak solution to the initial-boundary-value problem (2.1)-(2.3) with a coefficient $c\in C(\ooo{\OOO})$ by the unknown function
and with zero source function $F=F(x,t)\equiv 0$ and by $v$ the weak solution to (2.1)-(2.3)
with the coefficient $c-\Vert c\Vert_{C(\ooo{\OOO})}$ by the unknown function and  with zero source function $F=F(x,t) \equiv 0$. Because the inequalities $a(x) \ge 0$ and $c(x)-\Vert c\Vert_{C(\ooo{\Omega})}
\le 0$ hold true, for $v$ we can employ the results proved in  \cite{Liu} that say that
for an arbitrary but fixed $x \in \OOO$ there exists an at most finite set
$E_x$ such that
\vskip -10pt
$$
v(x,t) \le 0, \quad t \in E_x.   \eqno{(4.5)}
$$
\vskip -2pt \noindent
Since $c(x) \ge c(x) - \Vert c\Vert_{C(\ooo{\OOO})}$ in $\OOO$,
Theorem \ref{t2} leads to the inequality $u(x,t) \ge v(x,t)$ that together with the inequality (4.5)  completes the proof of Corollary \ref{c2}.
\end{proof} 

\vspace*{-3pt} 

\section{Conclusions and remarks} 

In this paper, we proved a weak maximum principle for the weak solution to an initial-boundary-value problem for a single-order time-fractional diffusion equation without a restriction on the sign of the coefficient $c=c(x)$ by the unknown function as well as some of its important consequences.  A result of this kind is well-known for the elliptic and parabolic type partial differential equations and for them the case of a bounded coefficient $c=c(x)$ is easily reduced to the case of a non-positive coefficient by constructing an auxiliary function with an exponential factor.  However, this technique does not work for the fractional diffusion equation. Instead, we reduced the problem with an arbitrary continuous coefficient to an integral equation for the solution to the problem with a negative coefficient and applied the fixed point theorem for the investigation of this equation.

From the maximum principle that we proved under weaker conditions  compared to those formulated in the already published papers (see the introduction for a short overview of the relevant publications),
a series of important consequences can be derived. In particular, we proved that the solution mapping $\{a, F\} \longrightarrow
u_{a,F}$ ($a$ and $F$ denote an initial condition and a source function of the problem under consideration, respectively, and $u_{a,F}$ denotes its weak solution)  preserves its sign. Moreover, the monotonicity of the solution with respect to the coefficient $c=c(x)$ by the unknown function has been shown. This solution property can be used among other things to characterize the set of
the points where the the weak solution can be non-positive.

It is worth mentioning that several other important results can be derived by the same arguments as we employed in the proof of
Theorem \ref{t1}. Let us briefly discuss one of them, namely, the maximum principle for a coupled system of the time-fractional diffusion equations with the fractional derivatives of the same order $\alpha$ $(0<\alpha <1$) in the form
$$
\fdif \left(
\begin{array}{c}
u_1(x,t) \\
\vdots   \\
u_N(x,t)
\end{array}\right)
= \Delta\left(
\begin{array}{c}
u_1(x,t) \\
\vdots   \\
u_N(x,t) \\
\end{array}\right)
$$
$$
+ \ \left(
\begin{array}{ccc}
p_{11}(x) & \cdots & p_{1N}(x) \\
\vdots   &  \vdots & \vdots    \\
p_{N1}(x) & \cdots & p_{NN}(x) \\
\end{array}\right) \ \
\left(
\begin{array}{c}
u_1(x,t) \\
\vdots   \\
u_N (x,t) \\
\end{array}\right)
+ \left(
\begin{array}{c}
F_1(x,t) \\
\vdots   \\
F_N(x,t) \\
\end{array}\right), \ x \in \OOO, \thinspace t>0.
$$
We assume that the inequalities
$p_{ij}(x) \ge 0$ hold true for $i\ne j$ a.e. in $\OOO$.

Then the inequalities $F_k(x,t)\ge 0$ a.e. in $\OOO\times (0,T)$ and
$u_k(x,0) \ge 0$ a.e. in $\OOO$ for $k=1,..., N$ yield the non-negativity of all solution components:
$$
u_k(x,t) \ge 0 \quad \mbox{a.e. in $\OOO\times (0,T)$ for $k=1,..., N$}.
$$
Our method can be also employed for derivation of the maximum principles for more general fractional differential equations
like e.g. the multi-term time-fractional diffusion equations, the diffusion equations of the distribute order and even for the general diffusion equations that were introduced and studied in \cite{Koch11} and \cite{luc-yam}.
These problems are worth to be considered and they will be studied elsewhere.

\vspace*{-6pt}  

\section*{Acknowledgements}

The first author Y.L. is partially supported by Bulgarian National Science Fund (Grant DFNI-I02/9).

The second 
author M.Y. is partially supported by Grant-in-Aid for Scientific
Research (S) 15H05740 of Japan Society for the Promotion of Science.


\end{document}